\documentclass[a4paper,11pt,english]{article}
\usepackage{babel}
\usepackage[latin1]{inputenc}

\usepackage{color}

\usepackage{authblk}

\author[1]{Thazin Aye}
\author[2]{Jian Fang\thanks{jfang@hit.edu.cn}}
\author[3]{Yingli Pan}

\affil[1]{Department of Mathematics, Harbin Institute of Technology, China}
\affil[2]{Institute	for Advanced Study in Mathematics and Department of Mathematics, Harbin Institute of Technology, China}
\affil[3]{Institute for Mathematical Sciences, Renmin University of China, China}

\title{\textbf{On a population model in discrete periodic habitat. I. spreading speed and optimal dispersal strategy\thanks{This work received fundings from the NSF of China (11771108) and the NSF of Heilongjiang Province of China (LC2017002). J. Fang would like to thank the Fields Institute for hospitality and support of the Fields Research Fellowship in August 2019, during which part of this work was done. }}}
\date{\today}
\usepackage{amsmath,amsfonts,amssymb}
\usepackage{yhmath,mathrsfs,yfonts,arcs}
\usepackage{amsthm}
\usepackage{graphicx}
\usepackage{enumerate}

\usepackage{lineno}

\usepackage{hyperref}
\hypersetup{colorlinks=true,       
    linkcolor=red,          
    citecolor=green,        
    filecolor=magenta,      
    urlcolor=cyan           
    }

\newtheorem{theorem}{Theorem}[section]
\newtheorem{remark}[theorem]{Remark}

\newtheorem{lemma}[theorem]{Lemma}

\newcommand{\field}[1] {\mathbb{#1}}

\newcommand{\Z}{\field{Z}}

\newcommand{\R}{\field{R}}

\newcommand{\f}{\frac}

\def\a{\alpha}
\def\b{\beta}
\def\e{\varepsilon}

\newcommand{\mc}{\mathcal}

\usepackage[textwidth=16.0cm,textheight=24cm]{geometry}

\begin{document}
\maketitle



\begin{abstract}
Starting from an age-structured diffusive population growth law for single species in a discrete and periodic habitat, we formulate a stage structured population model with spatially periodic dispersal, mortality and recruitment.  With a KPP type setting, after establishing the fundamental solution of a discretized heat equation with spatially periodic dispersal, we apply some recently developed dynamical system theories to obtain the existence of the spreading speed and its coincidence with the minimal speed of pulsating waves, as well as the variational characterization of the speed, by which we further analyze how the habitat periodicity influences the speed. In particular, there is a unique optimal dispersal strategy to maximize the speed. 
\end{abstract}

\smallskip
\noindent \textbf{Keywords:} periodic habitat, spreading speed, optimal dispersal strategy.

\smallskip
\noindent \textbf{2010 MSC: 35K57, 92D25} 

\setcounter{equation}{0}
\section{Introduction}

How does the spatial heterogeneity influence the species invasion is a challenging question. Many works have been devoted to the study of the dramatic influence of spatial heterogeneity on the complex invasive dynamics. Among the various kinds of heterogeneities, a typical one is the periodically fragmental habitat \cite{HR-TJM14}, such as corn and paddy fields. Recent studies also reveal that river bottom may provide a periodical habitat for the benthos due to the water drifts \cite{LLM-BMB06}. Traveling waves and spreading speed are two useful mathematical objectives for the study of invasion phenomena \cite{Fisher, KPP, AW-AM78}. In heterogeneous habitat, new mathematical objective arise, such as pulsating waves, generalized transition waves and global mean speed, see \cite{SKT-TPB86, Xin-SIAM00, Weinberger-JMB02, BH-CPAM02, BH-CPAM12} and references therein.

In this paper, we are interested in the scenario that how a periodic discrete habitat influences the invasion speed of a stage-structured species. For this purpose, we ideally assume that the one-dimensional discrete habitat $\Z$ can be classified into two classes: {\it even locations are referred to good and odd locations bad.} With such an assumption, we will formulate the following model
\begin{equation}\label{model}
\left\{
\begin{array}{ll}
u_i'(t)=\a [u_{i-1}(t)+u_{i+1}(t)]-2\b u_i(t)  -\gamma u_t(t)+f(u_i(t-\tau)),&  \text{$i$ is even},\\
u_i'(t)=\b [u_{i-1}(t)+ u_{i+1}(t)]-2\a u_i(t)  -\eta u_i(t),& \text{$i$ is odd},
\end{array}
\right.
\end{equation}
where $u_i(t)$ represents the density of the matured population at time $t$ and location $i$, $\a$ is the dispersal rate from bad locations to their adjacent good locations,  $\b$ is the dispersal rate from good locations to their adjacent bad locations, $\gamma$ and $\eta$ are the mortality rates, $f$ is the recruitment for mature population in good locations, and $\tau$ is the maturation age. The derivation details of \eqref{model} will be given in the next section. 

Models in discrete habitat are known as "patchy models", which had been widely established for the study of disease spread and species invasion, for instance we refer to \cite{Lou-07, Shen-98, FZ-JEMS15, Gao-19, LP-19, WWR-15}. Below we review some studies on patchy models of the Fisher-KPP type, by which our study is highly motivated. In 1993, Zinner, Harris and Hudson \cite{ZHH-JDE93} studied the invasion dynamics of the following model
\begin{equation}\label{Zinner}
u_i'=d(u_{i-1}-2u_i+u_{i+1})+u_i(1-u_i),
\end{equation}
which is a discretized version of the Fisher-KPP equation. Weng, Huang and Wu \cite{WHW-IMA03} introduced global interactions (induced by time delay) into \eqref{Zinner} by modeling the invasion of a stage-structured species:
\begin{equation}\label{WHW}
u_i'(t)=d(u_{i-1}(t)-2u_i(t)+u_{i+1}(t))-u_i(t)+\sum_{k\in\Z}\beta_kf(u_{i-k}(t-\tau)),
\end{equation}
where $\beta_k\ge 0$ and $\sum_{k\in\Z}\beta_k=1$. Guo and Hamel \cite{GH-MA06} investigated the spatially periodic Fisher-KPP equation in discrete habitat: 
\begin{equation}\label{GH}
u_i'=d_{i+1}(u_{i+1}-u_i)+d_i(u_{i-1}-u_i)+f_i(u_i),
\end{equation}
where $d_{j}=d_{i-N}$ for some $N>0$ and all $i\in\Z$. Wu and Hsu \cite{WH-DCDS18} combined the spatially periodic heterogeneity in \eqref{GH} and the global interactions in \eqref{WHW} to obtain a general model
\begin{equation}\label{WH}
u_i'(t)=d_{i+1}(u_{i+1}(t)-u_i(t))+d_i(u_{i-1}(t)-u_i(t))+f_i(u_i(t), \sum_{k\in\Z}\beta_k (u_{i-k}(t-\tau))).
\end{equation}
For \eqref{WHW}-\eqref{WH}, with suitable assumptions the Fisher-KPP structure can be verified, and the authors have proved the existence of the spreading speed and its coincidence of the minimal wave speed, as well as the variational characterization of the spreading speed\cite{WHW-IMA03, GH-MA06, WH-DCDS18}. Moreover, for \eqref{GH} the authors \cite{GH-MA06} also showed the convergence of the spreading speed in discrete habitat to that in continuous habitat in a certain sense. For \eqref{WH},  the authors \cite{WH-DCDS18} also showed that the periodicity in the recruitment term can increase the speed in a certain sense when diffusion is homogeneous (i.e. $d_i\equiv d, i\in\Z$) and $\tau=0$. The exponential stability of the pulsating waves is also obtained in \cite{WH-DCDS18}. But it is still unclear how the spatially periodic diffusion influences the speed.

In \eqref{model}, the species' dispersal is assumed to have a preference to good locations. And the resulting diffusion strategy  in \eqref{model} is different from \eqref{GH} and \eqref{WH}. With such a directional diffusion, we intend to study how it influences the invasion speed under a KPP type setting. In particular, we note that in bad locations, there is no birth. As such, large dispersal to bad locations (i.e., $\beta>\beta_0$ for some positive $\beta_0$) may lead to the species distinction, while no diffusion to bad locations (i.e., $\beta=0$) leads to disconnection of locations, which immediately implies invasion failure.  {\it This intuition then gives rise to the question that how to choose an appropriate dispersal strategy to maximize the invasion speed, which is increasing in both the population density $u_i$ and diffusion coefficient $\beta$, while $u_i$ is decreasing in $\beta$.}

Define
\begin{equation}
\Gamma:=\gamma+\f{2\beta\eta}{2\alpha+\eta}.
\end{equation}
Assume that $f\in C^1$ and there exists $w^*>0$ such that 
\begin{equation}\label{f1}
f(w)-\Gamma w
\begin{cases}
=0, & w=0\quad \text{or}\quad w^*,\\
>0, & w\in(0,w^*)
\end{cases}
\end{equation}
and
\begin{equation}\label{f2}
f'(w)\ge 0,\quad \left(f(w)/w\right)'\le 0, \quad w\in (0,w^*].
\end{equation}
From condition \eqref{f1} we see that $0$ is a steady state of \eqref{model}. From condition \eqref{f2} we see that $f$ is nondecreasing and sublinear, which, combining with condition \eqref{f1}, further implies that state $0$ is linearly unstable and there is a unique positive steady state
\begin{equation}
\text{$U^*:=\{u_i^*\}_{i\in \Z}$ with $u_i^*=u_{i+2}^*$ for $i\in\Z$.}
\end{equation}
A typical example satisfying \eqref{f1} and \eqref{f2} is $f(w)=\f{pw}{q+w}, p>\Gamma, q>0$. 

Recall that a positive number $c^*$ is the spreading speed of \eqref{model} provided that for any initial value $\phi\in C([-\tau,0]\times \Z, \R)$ with $0\le \phi(\theta,i)\le u_i^*$
and 
$\phi(\theta, i)\equiv 0$ when $(\theta, i) \in[-\tau,0]\times [-L,L]$ for some $L>0$,  the following limits hold: 
\begin{equation}
\lim_{t\to\infty}\max_{|i|\ge (c^*+\epsilon)t}|u(t,i)|=0=\lim_{t\to\infty}\max_{|i|\le (c^*-\epsilon)t}|u(t,i)-u_i^*|,\quad \forall \epsilon\in (0,c^*).
\end{equation}
A function $W(i,\xi), i\in\Z, \xi\in\R$ is a pulsating wave of \eqref{model} with average speed $c\in\R$ provided that  $W(i,i-ct)$ is a solution of \eqref{model} and
\begin{equation}
W(i,-\infty)=u_i^*,\quad W(i,+\infty)=0,\quad W(i,\xi)=W(i+2,\xi),\quad i\in\Z, \xi\in\R.
\end{equation}

\begin{theorem}\label{sstw}
Assume that \eqref{f1} and \eqref{f2} hold. Then \eqref{model} admits the spreading speed $c^*$, which coincides with the minimal speed of pulsating waves. Moreover, 
\begin{equation}
c^*=\min_{\mu>0}\f{\lambda(\mu)}{\mu},
\end{equation}
where $\lambda(\mu)$ is the unique positive solution of $-(\lambda+2\beta+\gamma)+\f{\alpha\beta(e^{\mu}+e^{-\mu})^2}{\lambda+2\alpha+\eta}+f'(0)e^{-\lambda\tau}=0$. 
\end{theorem}
By Theorem \ref{sstw}, we see that with the KPP structure if the population growth is big enough ($f'(0)>\Gamma$), then the species can spread out and the spreading speed coincides with the minimal speed of the pulsating waves. 

By conditions \eqref{f1} and \eqref{f2}, we infer that  $f'(0)=\lim_{w\to0}\f{f(w)}{w}\ge  \f{f(\f{w^*}{2})}{\f{w^*}{2}}>\Gamma.$
In view of the definition of $\Gamma$, we see that $f'(0)>\Gamma$ is equivalent to $\beta<\beta_0$, where $\beta_0$ is defined by
\begin{equation}\label{def-beta0}
\beta_0:=\f{(f'(0)-\gamma)(2\alpha+\eta)}{2\eta}.
\end{equation}

\begin{theorem}\label{ss1}
There exists $\beta_1\in(0,\beta_0)$ such that $c^*=c^*(\beta)$ increases in $\beta\in(0,\beta_1)$ and decreases in $\beta\in(\beta_1,\beta_0)$. Further, 
\begin{equation}\label{maximum}
\max_{\beta\in(0,\beta_0)} c^*(\beta)=c^*(\beta_1)=\f{\lambda^*}{\cosh^{-1}\left(\sqrt{\f{\lambda^*+2\alpha+\eta}{2\alpha}}\right)},
\end{equation}
where $\lambda^*$ is the unique positive solution of the transcendental equation $-(\lambda+\gamma)+f'(0)e^{-\lambda\tau}=0$.  
\end{theorem}
By Theorem \ref{ss1}, we see that bad locations are not good for the population growth, but to spread out the species has to disperse to bad locations, scarificing a part of the population. To maximize the invasion speed, the species has an optimal dispersal strategy in terms of $\beta=\beta_1$. Further, for the classical KPP equation $u_t=u_{xx}+f(u)$,  it is well-known that the spreading speed is $2\sqrt{f'(0)}$, which has the same order as $\sqrt{f'(0)}$ when $f'(0)\to+\infty$. But  for the case we studied in spatially periodic habitat, if the species always choose the optimal dispersal strategy $\beta=\beta_1$  then the speed has a higher order than $\sqrt{f'(0)}$ when $f'(0)\to\infty$. Indeed, $\beta_1$ depends on $f'(0)$ and $\lim_{f'(0)\to\infty}\beta_1=\infty$. By \eqref{maximum} we can infer that
\begin{equation}\label{eq901}
\lim_{f'(0)\to\infty} \f{\ln \lambda^*}{\lambda^*}c^*(\beta_1)=2,
\end{equation}
where $\lambda^*$ increases in $f'(0)$ with $\lim_{f'(0)\to\infty}\lambda^*=\infty$. Meanwhile, assuming $\tau=0$ for the sake of simplicity, we have $\lambda^*=f'(0)-\gamma$, and hence, 
\begin{equation}\label{eq902}
\liminf_{f'(0)\to\infty} \f{\ln \lambda^*}{\lambda^*} \sqrt{f'(0)}= \liminf_{x\to\infty} \f{\ln x}{\sqrt{x}}=0. 
\end{equation}
Therefore, by \eqref{eq901} and \eqref{eq902} we see that if the initial growth rate is sufficiently large (i.e., $f'(0)\gg 1$) and $\tau$ is small, then the species may sacrifice a large number of population to reach a spreading speed having higher order than $\sqrt{f'(0)}$, by taking the advantage of habitat periodicity. This is not seen in heterogeneous habitat.  For the case $\tau>0$, one may obtain a similar conclusion.

\begin{remark}
$\beta_1$ is the unique zero of an implicit function, see \eqref{beta-eq}, which provides a way to numerically calculate the optimal dispersal rate. 
\end{remark}

Note that $\Gamma=\Gamma(\eta)$ decreases in $\eta$ and $\lim_{\eta\to\infty}\Gamma(\eta)=2\beta+\gamma$. Define
\begin{equation}
\eta_0:=
\begin{cases}
 \f{2\alpha(f'(0)-\gamma)}{2\beta+\gamma-f'(0)}, & \text{if $f'(0)\in (\Gamma, 2\beta+\gamma)$},\\
+\infty, & \text{if $f'(0)\ge 2\beta+\gamma$}.
\end{cases}
\end{equation}
\begin{theorem}\label{ss2}
$c^*=c^*(\eta)$ decreases in $\eta$ to zero as $\eta$ increases to $\eta_0$.
\end{theorem}

By  Theorems \ref{ss2} we see that if the population growth is big enough ($f'(0)>2\beta +\gamma$), then no matter how big the death rate $\eta$ in bad locations is, the species can always successfully invade, though the speed decreases in $\eta$.

In a companion paper \cite{AFP2}, we study the scenario that a strong Allee effect is assumed in birth. Then a bistable structure and propagation failure may appear if the diffusion rate $\beta$ is in appropriate ranges. In another companion paper \cite{ASS}, the first author and her collaborators studied the bifurcation dynamics of the model \eqref{model} when the birth function is of unimodal type.

The rest of this paper is organized as follows. In section 2, model derivation details are presented. In section 3, the fundamental solution of a discretized heat equation with periodic diffusion is obtained. Sections 4 and 5 are devoted to the proofs of the main theorems.

\setcounter{equation}{0}
\section{Model formulation and preliminary}
Let $\rho(t,i,a)$ be the population density of the species with age $a$ and at time $ t\ge 0 $ and location $ i\in \Z $. Assume that the evolution of the species obeys the following growth law:
\begin{equation}
\left(\frac{\partial}{\partial t}+\frac{\partial}{\partial a} \right)\rho(t,i,a)=\sum_{j\neq i,j\in\Z}d_{ij}(a)\rho(t,j,a)-
\sum_{j\neq i,j\in\Z}d_{ji}(a)\rho(t,i,a)-r(i,a)\rho(t,i,a),\quad t>0, i\in\Z,
\end{equation}
where $r(i, a) $ is the death rate at location $i$, and $d_{ji}(a)$ is the diffusion rate from location $i$ to location $j$. The derivation of diffusive delayed population models based on the age-structured growth law may go back to the work \cite{SWZ-PRS01} . But for readers' convenience, we give the details below.  

The biological scenario of interest is the one-dimensional periodic habitat $\Z$. For this purpose, we ideally divide $\Z$ into two classes. {\it If $i$ is even, then we call location $i$ good; if $i$ is odd, then we call location $i$ bad.} We assume that population has very distinct behaviors in good and bad locations, which will be specified later.

Let $\tau>0$ be the maturation time of the species. Then the population can be classified into two stages by age: mature and immature. An individual is assumed to belong to mature stage if and only if its age is not less than $\tau$. Hence, the integrand
\begin{equation}
u_i(t)=\int_\tau^\infty\rho(t,i,a)da
\end{equation}
denotes the density of mature population at time $t$ and location $i$.

We make the following biological assumptions:
\begin{enumerate}
\item[(A1)]All individuals in the same stage and the same class of locations share the same characteristics and behaviors;
\item[(A2)]Diffusion only happens in mature stage and symmetrically between adjacent locations;
\item[(A3)]There are no newborns in bad locations.
\end{enumerate}
By (A1), we accordingly assume that
\begin{equation}
d_{i,j}(a)=
\begin{cases}
d_{i,j}^M, & a\ge\tau,\\
d_{i,j}^I, & 0< a<\tau,
\end{cases}
\quad 
r(i, a)=
\begin{cases}
r_i^M, & a\ge\tau,\\
r_i^I, & 0< a<\tau,
\end{cases}
\end{equation}
where $d_{i,j}^M, d_{i,j}^I$ and $r_i^M, r_i^I$ are some constants depending on locations. Combining with (A2), we further assume that
\begin{equation}\label{diffusion}
d_{i,j}^I=0, \forall i,j\in\Z,
\quad
d_{i,j}^M=
\begin{cases}
0, & |i-j|\ge 2,\\
\beta^M, & \text{$i$ is odd and $j=i\pm 1$},\\
\alpha^M, & \text{$i$ is even and $j=i\pm 1$}
\end{cases}
\end{equation}
and
\begin{equation}\label{death}
r_i^M=
\begin{cases}
\eta^M, & \text{$i$ is odd}, \\
\gamma^M, & \text{$i$ is even},
\end{cases}
\quad
r_i^I=
\begin{cases}
\eta^I, & \text{$i$ is odd}, \\
\gamma^I, & \text{$i$ is even},
\end{cases}
\end{equation}
where $\alpha^M, \beta^M, \gamma^M, \eta^M, \gamma^I$ and $\eta^I$ are all positive constants.  By (A3), we assume that
\begin{equation}\label{newborn}
\rho(t,i,0)=
\begin{cases}
0, & \text{$i$ is odd},\\
b(u_i(t)), & \text{$i$ is even},
\end{cases}
\end{equation}
where $b(s)=pse^{-qs}$ is the Ricker type birth function. With the aforementioned assumptions, we differentiate the density $u(t,i)$ of mature population with respect to time $t$, yielding

\begin{equation}\label{orginal-eq}
\begin{small}
\begin{aligned}
&u_i'(t)=\int_{\tau}^{+\infty}\frac{\partial}{\partial{t}}\rho(t,i,a)da \\
&=\int_{\tau}^{+\infty}\!\!\!\left[-\frac{\partial}{\partial{a}}\rho(t,i,a)\!-\!\sum\limits_{j\neq{i},j\in\Z}d^M_{ji}(a)\rho(t,i,a)\!\!+\!\!\sum\limits_{j\neq{i},j\in\Z}d^M_{ij}(a)\rho(t,j,a)\!\!-\!\!r^M_i\rho(t,i,a)\right]da\nonumber\\
&=
\begin{cases}
\rho(t,i,\tau)+\beta^M[u_{i+1}(t)+u_{i-1}(t)]-2\alpha^M u_i(t)-\eta^M u_i(t),& \text{$i$ is odd}\\
\rho(t,i,\tau)+\alpha^M[u_{i+1}(t)+u_{i-1}(t)]-2\beta^M u_i(t)-\gamma^M u_i(t), & \text{$i$ is even},
\end{cases}
\end{aligned}
\end{small}
\end{equation}
where  the biologically reasonable assumption $\rho(t,i,\infty)=0 $ was made and used. To get a closed form of the model, we next calculate $\rho(t,i,\tau)$ in term of $u(t, i)$ in a certain way. Indeed, $\rho(t,i,\tau)$ represents the newly matured population at time $t$ and location $i$. It is the evolution result of newborns at $t-\tau$ and location $i$ since immature population does not move. That is, there is an evolution relation between the quantities $\rho(t,i,\tau)$ and $\rho(t-\tau,i,0)$. More precisely, the relation is the time-$\tau$ solution map of the following evolution equation
\begin{equation}\label{evolution eq}
    \left\{
\begin{aligned}
&\frac{\partial}{\partial s}z(s,i)=-r_i^Iz(s,i),0\leq s\leq \tau\\
&z(0,i)=\rho(t-\tau,i,0),
\end{aligned}
    \right.
\end{equation}
which, combining with \eqref{diffusion}-\eqref{newborn}, implies that
\begin{equation}
\rho(t,i,\tau)=
\begin{cases}
0, & \text{$i$ is odd},\\
e^{-\gamma^I \tau} b(u_i(t-\tau)), & \text{$i$ is even}.
\end{cases}
\end{equation}
Consequently, we obtain the following system modeling the mature population in discrete periodic habitat subject to the biological assumptions (A1)-(A3):
\begin{equation}\label{model1}
\begin{cases}
u_i'(t)= \beta[u_{i+1}(t)+u_{i-1}(t)]-2\alpha u_i(t)-\eta u_i(t),& \text{$i$ is odd},\\
u_i'(t)= \alpha[u_{i+1}(t)+u_{i-1}(t)]-2\beta u_i(t)-\gamma u_i(t)+\mu b(u_{i}(t-\tau)), & \text{$i$ is even},
\end{cases}
\end{equation}
where the superscript $M$ was dropped and
\begin{equation}\label{def-mu}
\mu=\mu(\tau):=e^{-\gamma^I \tau}
\end{equation}
is the survival rate from newborn to being adult in good locations. Setting $\mu b(u)=f(u)$ we obtain the model \eqref{model}.

%

Let $\R_+=[0,+\infty)$. Define 
\begin{equation}
X=C([-\tau,0], \R),\quad X_+=C([-\tau, 0], \R_+).
\end{equation} 
For $u\in X$, define 
\begin{equation}
\|u\|_X=\max_{\theta\in[-\tau, 0]}|u(\theta)|.
\end{equation} 
For any $u$ and $v$ in $X$ we write $u\ge v$ if $u-v\in X_+$, $u>v$ if $u\ge v$ but $u\not\equiv v$, and $u\gg v$ if $u-v\in \text{Int} X_+$, where $\text{Int} X_+$ denotes the interior of $X_+$. Then $(X,X_+,\|\cdot\|)$ is a Banach lattice. For $r_2>r_1$ in $X$, define the order interval $[r_1,r_2]_{X}$ by 
\[
[r_1,r_2]_X:=\{u\in X: r_2\ge u\ge r_1 \}.
\]
For the sake of convenience, we often write $X_{r}$ instead of $[0,r]_{X}$ for $r>0$ in $X$.

Let $\mc{C}$ be the space of all uniformly bounded functions from $\Z$ to $X$. We equip $\mc{C}$ with the compact open topology, that is, $\phi_n$ converges to $\phi$ in $\mc{C}$ if and only if $\phi_n(i)$ converges to $\phi(i)$ in $X$ for each $i\in\Z$. Such a topology can be induced by the following norm
\begin{equation}
\|\phi\|_{\mc{C}}:=\sum_{ k=1}^\infty \f{\max_{|i|\le k}|\phi(i)|_{X}}{2^k}.
\end{equation}
For any $\phi$ and $\psi$ in $\mc{C}$, we write $\phi\ge \psi$ if $\phi(i)\ge \psi(i)$ for all $i\in\Z$, $\phi>\psi$ if $\phi\ge \psi$ but $\phi\neq \psi$, and $\phi\gg\psi$ if there exists $\Gamma=\Gamma(\phi,\psi)>1$ such that $\phi(i)>\Gamma \psi(i)$ for all $i\in\Z$. For $\beta_2>\beta_1$ in $\mc{C}$, we define the order interval $[\beta_1,\beta_2]_\mc{C}$ by
\begin{equation}
[\beta_1,\beta_2]_\mc{C}:=\{\phi\in \mc{C}: \beta_2 \ge \phi\ge \beta_1\}. 
\end{equation}
For the sake of convenience, we often write $\mc{C}_{\beta}$ instead of $[0,\beta]_{\mc{C}}$ for $\beta>0$ in $\mc{C}$.

Define
\begin{equation}\label{def-O}
\mc{O}:=\{\text{all uniformly bounded functions from $\Z$ to $\R$ }\}.
\end{equation} 
Since $\R\subset X$, we regard $\mc{O}$ as a subspace of $\mc{C}$, inheriting the topology and ordering of $\mc{C}$.

In the following we define a class of infinite dimensional matrices and their actions on $\mc{O}$ in the way that we will used. For any sequence of nonnegative real numbers $p_{i,j}, i,j\in\Z$ with 
\begin{equation}\label{c201}
\sum_{j\in\Z}p_{i,j}<+\infty, \quad \text{uniformly in $i\in \Z$,}
\end{equation} 
we define $P:\mc{O}\to \mc{O}$ by 
\begin{equation}
P[\phi](i)=\sum_{j\in \Z}p_{i,j}\phi(j).
\end{equation}
For such $p_{i,j}$ and $P$, we will use $(P)_{i,j}$ to denote $p_{i,j}$ and $(p_{i,j})_{i,j\in\Z}$ to denote $P$ for the sake of convenience. For such two operators $P=(p_{i,j})_{i,j\in\Z}$ and $Q=(q_{i,j})_{i,j\in\Z}$, we may define $PQ: \mc{O}\to \mc{O}$ by $(PQ)_{i,j}=\sum_{k\in\Z}p_{i,k}q_{k,j}$.

Let
\begin{equation}\label{def-matrix}
a_{ij}=
\begin{cases}
\beta, &\text{ $j=i\pm 1$ and $i$ is even}\\
\alpha, &\text{ $j=i\pm 1$ and $i$ is odd}\\
0,& \text{elsewhere}
\end{cases}
\quad \text{and} \quad
b_{ij}=
\begin{cases}
-2\b-\gamma, &\text{ $j=i$ is even}\\
-2\a-\eta, &\text{ $j=i$ is odd}\\
0,& \text{elsewhere}.
\end{cases}
\end{equation}
Clearly, both $a_{i,j}$ and $b_{i,j}$ satisfy \eqref{c201}. Hence, we may similarly define linear operators $A$ and $B$ from $\mc{O}\to \mc{O}$ by
\begin{equation}\label{linear operators}
A[\phi](i)=\sum_{j\in\Z} a_{ij}\phi(j),\quad B[\phi](i)=\sum_{j\in\Z} b_{ij}\phi(j).
\end{equation}
Further, we may define the $n$-th iterations of $A$ and $B$, respectively. For example, 
\begin{equation}
(A^2)_{i,j}=\sum_{k\in\Z} (A)_{i,k}(A)_{k,j}=\sum_{k\in\Z} a_{i,k}a_{k,j}.
\end{equation}
For $t\ge 0$, define $U(t)\in \mc{O}$ with 
\begin{equation}
U(t)(i)=u_i(t),\quad i\in\Z.
\end{equation}
Define $F:\mc{O}\to\mc{O}$ by
\begin{equation}
F[U(t)](i):=
\begin{cases}
f(u_i(t)),& \text{$i$ is even},\\
0,& \text{$i$ is odd}.
\end{cases}
\end{equation}
Then, with these notations we can write \eqref{model} as the following form:
\begin{equation}\label{abstract-eq}
U'(t)=A[U(t)]+B[U(t)]+F(U(t-\tau)),
\end{equation}
which consists of countably many coupled delayed differential equations.

\section{Fundamental solution matrix of $U'=AU$}

The linear system of countably many ordinary differential equations $U'=AU$ can be regarded as a discrete analogue of $u_t=\f{1}{d(x)}(d(x)u_x)_x, x\in\R$ with periodic diffusion coefficient $d(x)$. The purpose of this section is to find the fundamental solution matrix of $U'=AU$.

We will extend the concept of matrix exponential from finite to infinite dimensional matrix. Surprisingly it is not as obvious as we expect. We first formally define $e^{tA}$ using the addition of a series of matrices. Then we prove that $e^{tA}$ maps $\mc{O}$ to $\mc{O}$, where $\mc{O}$ is defined as in \eqref{def-O}. Finally we show that $e^{tA}$ is linear, strongly positive and continuous with respect to the compact open topology. Further, it is the fundamental solution matrix of $U'=AU$.

Define formally $e^{tA}$ by 
\begin{equation}\label{def-etA}
e^{tA}[\phi](i):=\sum_{j\in\Z} \sum_{n=0}^\infty \f{t^n (A^n)_{i,j}}{n!} \phi(j)
\end{equation}
and $(e^{tA})_{i,j}$ by
\begin{equation}
(e^{tA})_{i,j}:=\sum_{n=0}^\infty \f{t^n (A^n)_{i,j}}{n!},
\end{equation}
where $A^0$ is naturally understood as the identity map from $\mc{O}$ to $\mc{O}$. More precisely, $(A^0)_{i,j}=1$ if $i=j$ and $(A^0)_{i,j}=0$ if $i\neq j$. By the definition of $A$ and $A^n$, we can infer that 
\begin{equation}
 (A^n)_{i,j}=\sum_{k=i\pm 1} (A)_{i,k}(A^{n-1})_{k,j}\le \sup\{\alpha,\beta\} [(A^{n-1})_{i-1,j}+(A^{n-1})_{i+1,j}],
\end{equation}
and inductively, 
\begin{equation}
0\le (A^n)_{i,j}\le 2^n(\sup\{\alpha,\beta\})^n,
\end{equation}
from which we see that $\sum_{n=0}^\infty \f{t^n (A^n)_{i,j}}{n!}<+\infty$ and $(e^{tA})_{i,j}$ is well-defined for any $i,j\in\Z$ and $t\ge 0$. 

The main result of this section is as follows.
\begin{theorem}\label{fundamental}
For $t\ge 0$, $e^{tA}$ is a linear, strongly positive and continuous map from $\mc{O}$ to $\mc{O}$. Further, $e^{tA}$ is the fundamental  solution matrix of $U'=AU$.
\end{theorem}

Such a result has its own interest and will also be essentially helpful to define the solution semiflow of the nonlinear problem \eqref{abstract-eq} in the next section, where we apply some dynamical system theories to study the propagation dynamics of \eqref{abstract-eq}.

To prove Theorem \eqref{fundamental}, we proceed with a series of lemmas. The first one is the explicit expression of $(A^{n})_{ i,j}$.

\begin{lemma}\label{Anij}
$(A^{n})_{ i,j}, n\ge 1$ has the following expression:
\begin{equation}\label{An}
(A^{n})_{ i,j}=
\left(\f{\b}{\a}\right)^\f{(-1)^{i+n}+(-1)^i}{4}(\a\b)^{\f{n-1}{2}}\displaystyle\sum_{m=0}^{n-1}\binom{n-1}{m}a_{ i-(n-1-2m),j}, \quad i,j\in\Z.
\end{equation}	
\end{lemma}
\begin{proof}
We proceed with the induction argument. 

We first prove the case where $n$ is odd. Since $(A)_{ i,j}=a_{i,j}, i,j\in\Z$, it then follows that \eqref{An} holds for $n=1$. Fix $k \in \Z$ with $k\ge 1$. Suppose \eqref{An} holds for $n=2k-1$, i.e., 
\begin{equation}\label{eq300}
(A^{2k-1})_{i,j}=(\a\b)^{k-1}\displaystyle\sum_{m=0}^{2k-2}\binom{2k-2}{m}a_{ i-(2k-2-2m),j}.
\end{equation}
Below we prove \eqref{An} also holds for $n=2k+1$. Note that
\begin{equation}\label{eq301}
(A^{2k+1})_{ i,j}=(A^{2k-1}A^{2})_{ i,j}=\sum_{r\in\Z} (A^{2k-1})_{i,r}(A^{2})_{r,j}
\end{equation}
and
\begin{equation}\label{eq302}
(A^{2})_{r,j}=\sum_{s\in\Z}a_{r,s}a_{s,j}=a_{r,r-1}a_{r-1,j}+a_{r,r+1}a_{r+1,j}=a_{r,r-1}[a_{r-1,j}+a_{r+1,j}]
\end{equation}
in virtue of the definition of $a_{i,j}$ in \eqref{def-matrix}. Combining \eqref{eq300}-\eqref{eq302} leads to
\begin{equation}\label{eq303}
(A^{2k+1})_{ i,j} =\sum_{r\in\Z} (\a\b)^{k-1}\sum_{m=0}^{2k-2}\binom{2k-2}{m}a_{ i-(2k-2-2m),r}a_{r,r-1}[a_{r-1,j}+a_{r+1,j}].
\end{equation}
By again the definition of $a_{i,j}$, we see that 
\begin{eqnarray}
a_{ i-(2k-2-2m),r}a_{r,r-1}=
\begin{cases}
\alpha\beta, &r=i-(2k-2-2m)\pm 1,\\
0, & \text{elsewhere},
\end{cases}
\end{eqnarray}
which, combining with \eqref{eq303}, implies that
\begin{eqnarray}\label{eq304}
&&(A^{2k+1})_{ i,j}\nonumber\\
= &&(\a\b)^{k}\sum_{m=0}^{2k-2}\binom{2k-2}{m} \sum_{r=i-(2k-2-2m)\pm 1}[a_{r-1,j}+a_{r+1,j}] \nonumber\\
=&&(\a\b)^{k}\sum_{m=0}^{2k-2}\binom{2k-2}{m} [a_{i-(2k-2m),j}+ 2a_{i-(2k-2-2m),j}+a_{i-(2k-4-2m),j}],
\end{eqnarray}
which consists of a linear combination of the following terms:
\[
a_{i-(2k-2s),j}, \quad  \text{$s, k\in\Z$ with $0\le s\le 2k$}.
\]
Reorganizing the order of sums in \eqref{eq304} we obtain the coefficients (denoted by $c_s$) for $a_{i-(2k-2s),j}, 0\le s\le 2k$. They read
\[
c_s(\a\b)^{-k}=
\begin{cases}
\binom{2k-2}{0}, & s=0,\\
\binom{2k-2}{1}+2\binom{2k-2}{0},&s=1,\\
\binom{2k-2}{s}+2\binom{2k-2}{s-1}+\binom{2k-2}{s-2}, &2\le s\le 2k-2,\\
2\binom{2k-2}{2k-2}+\binom{2k-2}{2k-3},&s=2k-1,\\
\binom{2k-2}{2k-2}, &s=2k.
\end{cases}
\]
After computing these binomial coefficients we arrive at
\begin{equation}
c_s(\a\b)^{-k}=\binom{2k}{s}, \quad 0\le s\le 2k,
\end{equation}
and hence, 
\begin{equation}
(A^{2k+1})_{ i,j}=(\a\b)^{k}\sum_{s=0}^{2k}\binom{2k}{s}a_{i-(2k-2s),j},
\end{equation}
which is exactly \eqref{An} with $n=2k+1$.

Next we prove \eqref{An} when $n$ is even. From \eqref{eq302} and the definition of $a_{i,j}$ we see that
\begin{equation}
(A^{2})_{i,j}=a_{i,i-1}[a_{i-1,j}+a_{i+1,j}]=\begin{cases}
\a(a_{i-1,j}+a_{i+1,j}), & \text{$i$ is odd;}\\
\b(a_{i-1,j}+a_{i+1,j}), &\text{$i$ is even}.
\end{cases}
\end{equation}
Therefore, \eqref{An}  is true for $n=2.$  Fix $k \in \Z$ with $k\ge 1$. Suppose \eqref{An} holds for $n=2k$, that is, 
\begin{equation}
(A^{2k})_{i,j}=\left(\f{\b}{\a}\right)^\f{(-1)^i}{2}(\a\b)^{\f{2k-1}{2}}\displaystyle\sum_{m=0}^{2k-1}\binom{2k-1}{m}a_{ i-(2k-1-2m),j}.
\end{equation}
Below we prove it is also true for $n=2k+2$.  Indeed, note that
\begin{eqnarray}\label{eq3010}
&&(A^{2k+2})_{ i,j}\nonumber\\
=&&(A^{2k}A^{2})_{ i,j}\nonumber\\
=&& \sum_{r\in \Z}(A^{2k})_{i,r}(A^{2})_{r,j}\nonumber\\
=&&\sum_{r\in \Z}\left(\f{\b}{\a}\right)^\f{(-1)^i}{2}(\a\b)^{\f{2k-1}{2}}\displaystyle\sum_{m=0}^{2k-1}\binom{2k-1}{m}a_{ i-(2k-1-2m),r}a_{r,r-1}[a_{r-1,j}+a_{r+1,j}]
\end{eqnarray}
From the definition of $a_{ij}$, we have
\[
a_{ i-(2k-1-2m),r}a_{r,r-1}=
\begin{cases}
\a\b, & r=i-(2k-1-2m)\pm 1,\\
0, & \text{elsewhere}.
\end{cases}
\]
Hence, \eqref{eq3010} becomes 
\begin{small}
\begin{eqnarray*}
&&(A^{2k+2})_{ i,j}\nonumber\\
=&&\left(\f{\b}{\a}\right)^\f{(-1)^i}{2}(\a\b)^{\f{2k+1}{2}}\displaystyle\sum_{m=0}^{2k-1}\binom{2k-1}{m}\sum_{r=i-(2k-1-2m)\pm 1}[a_{r-1,j}+a_{r+1,j}]\nonumber\\
=&&\left(\f{\b}{\a}\right)^\f{(-1)^i}{2}(\a\b)^{\f{2k+1}{2}}\displaystyle\sum_{m=0}^{2k-1}\binom{2k-1}{m}[a_{i-(2k-2m),j}+2a_{i-(2k-2-2m),j}+a_{i-(2k-4-2m),j}],
\end{eqnarray*}
\end{small}
which consists of a linear combination of the following terms:
\[
a_{i-(2k+1-2s),j}, \quad \text{$s,k\in\Z$ with $0\le s\le 2k+1$}.
\]
Reorganizing the order of sums in \eqref{eq306} we obtain the coefficients (denoted by $d_s$) for $a_{i-(2k+1-2s),j}$ with $0\le s\le 2k+1$. They read
\[
d_s= \left(\f{\b}{\a}\right)^{\f{(-1)^i}{2}}(\a\b)^{\f{2k+1}{2}}
\begin{cases}
\binom{2k-1}{0}, & s=0,\\
\binom{2k-1}{1}+2\binom{2k-1}{0},&s=1,\\
\binom{2k-1}{s}+2\binom{2k-1}{s-1}+\binom{2k-1}{s-2}, &2\le s\le 2k-1,\\
2\binom{2k-1}{2k-1}+\binom{2k-1}{2k-2},&s=2k,\\
\binom{2k-1}{2k-1}, &s=2k+1.
\end{cases}
\]
After computing these binomial coefficients we arrive at
\begin{equation}
d_s=\left(\f{\b}{\a}\right)^\f{(-1)^i}{2}(\a\b)^{\f{2k+1}{2}}\binom{2k+1}{s}, \quad 0\le s\le 2k+1,
\end{equation}
and hence, 
\begin{equation}
(A^{2k+2})_{ i,j}=\left(\f{\b}{\a}\right)^\f{(-1)^i}{2}(\a\b)^{\f{2k+1}{2}}\displaystyle\sum_{s=0}^{2k+1}\binom{2k+1}{s}a_{i-(2k+1-2s)},
\end{equation}
which is exactly \eqref{An} with $n=2k+2$.
\end{proof}

With the expression of $A^n_{i,j}$, we can give a bound for $(e^{tA})_{i,j}$.

\begin{lemma}\label{uppbd}
There exists $C_1,C_2>0$, depending only on $\alpha$ and $\beta$,  such that 
\begin{equation}\label{upp1}
0<(e^{tA})_{i,i}\le 1+C_1t e^{\alpha\beta t^2}+C_2(e^{\alpha\beta t^2}-1),\quad i\in\Z, t>0
\end{equation}
and
\begin{equation}\label{upp2}
0<(e^{tA})_{i+l,i}\le  C_1 t^2 \sum_{k\ge \f{|l|-1}{2}} \f{t^{2k}}{k!}(\a\b)^{k} +C_2  \sum_{k\ge \f{|l|}{2}} \f{t^{2k}}{k!}(\a\b)^k,\quad i\in\Z, l\in \Z\setminus\{0\}, t>0.
\end{equation}
\end{lemma}
Before the proof, we remark that such a bound suggests that $e^{tA}$ tends to the identity map in a certain sense as $t\to0$ and $e^{tA}$ is strongly positive.

\begin{proof}
We first show that $(e^{tA})_{i,j}>0$ for any $i,j\in\Z$ and $t>0$. Indeed, from \eqref{An} and \eqref{def-matrix} we can infer that $(A^n)_{i,j}\ge 0$ for any $i,j,n$ and $(A^n)_{i,j}>0$ for $n=|i-j|$. Therefore, $(e^{tA})_{i,j}=\sum_{n\ge 0}\f{t^n}{n!}(A^n)_{i,j}\ge \f{t^n}{n!}(A^n)_{i,j}|_{n=|i-j|}>0$.

Next we derive the upper bound for $(e^{tA})_{i,j}$. From the expressions of  $(A^n)_{i,j}$ obtained in Lemma \ref{Anij} we can infer that 
\begin{eqnarray}
0 \le (A^{n})_{i,j} \le (\a\b)^{\f{n-1}{2}}\max \left\{ \sqrt{\f{\b}{\a}},\sqrt{\f{\a}{\b}}\right\}\displaystyle\sum_{m=0}^{n-1}\binom{n-1}{m}a_{i-(n-1-2m),j}.
\end{eqnarray}
By the definition of $a_{i,j}$ we see that 
\begin{equation}
a_{i-(n-1-2m),j}=0, \quad \text{for $|i-j|>n$ and $0\le m\le n-1$}.
\end{equation}
Consequently, we have 
\begin{equation}
\sum_{m=0}^{n-1}\binom{n-1}{m}a_{i-(n-1-2m),j} =0,  \quad \text{when $|i-j|>n$}
\end{equation}
and
\begin{equation}
\sum_{m=0}^{n-1}\binom{n-1}{m}a_{i-(n-1-2m),j} \le 2 \max\{\alpha,\beta\} \max_{0\le m\le n-1}\binom{n-1}{m},  \quad \text{when $|i-j|\le n$}.
\end{equation}
Note that 
\begin{eqnarray}\label{ineq211}
\max_{0\le m\le n-1}\binom{n-1}{m}=\begin{cases}
\binom{n-1}{\f{n-1}{2}}, & \text{$n-1$ is even}\\
\binom{n-1}{\f{n}{2}}, & \text{$n-1$ is odd}.
\end{cases}
\end{eqnarray}
Define 
\begin{equation}
C_1:=2 \max \left\{ \sqrt{\f{\b}{\a}},\sqrt{\f{\a}{\b}}\right\}\max\{\alpha,\beta\}.
\end{equation} 
Then for $k\ge 1$ we obtain
\begin{equation}\label{ineq212}
(A^{2k})_{i,j}
\begin{cases}
=0, & |i-j|>2k,\\
\le  C_1 (\a\b)^{\f{2k-1}{2}}\binom{2k-1}{k}, & |i-j|\le 2k
\end{cases}
\end{equation}
and 
\begin{equation}\label{ineq213}
 (A^{2k-1})_{i,j}\begin{cases}
=0, & |i-j|>2k-1,\\
\le  C_1 (\a\b)^{k-1}\binom{2k-2}{k-1}, & |i-j|\le 2k-1
\end{cases}
\end{equation}

In the following, we proceed with two cases: (i) $i=j$; (ii) $i\neq j$.  

Case (i). Note that
\begin{equation}
(e^{tA})_{i,i}=\sum_{n\ge 0} \f{t^n}{n!} (A^n)_{i,i}=1+\sum_{k\ge 1}\f{t^{2k}}{(2k)!} (A^{2k})_{i,i}+\sum_{k\ge 1}\f{t^{(2k-1)}}{(2k-1)!} (A^{2k-1})_{i,i}.
\end{equation}
Combining with inequalities \eqref{ineq212}, \eqref{ineq213} and
\begin{equation}\label{ineqnew1}
\f{1}{(2k)!}\binom{2k-1}{k}\le \f{1}{k!},\quad \f{1}{(2k-1)!}\binom{2k-2}{k-1}\le \f{1}{(k-1)!},
\end{equation}
we obtain
\begin{eqnarray}
(e^{tA})_{i,i} &\le& 1+C_1(\a\b)^{-\f{1}{2}} \sum_{k\ge 1}\f{t^{2k}}{k!}(\a\b)^k  +C_1 t \sum_{k\ge 1} \f{t^{2k-1}}{(k-1)!}(\a\b)^{k-1}\nonumber\\
&\le & 1+C_1(\a\b)^{-\f{1}{2}} (e^{\alpha\beta t^2}-1)+C_1 t e^{\alpha\beta t^2}.
\end{eqnarray}
Denoting $C_1(\a\b)^{-\f{1}{2}}$ by $C_2$ we obtain \eqref{upp1}.  

Case (ii). Using the first part of inequalities \eqref{ineq211} and \eqref{ineq213}, for $l\neq 0$ we obtain
\begin{equation}
(e^{tA})_{i+l,i}=\sum_{n\ge 0} \f{t^n}{n!} (A^n)_{i+l,i}=\sum_{k\ge |j|/2}\f{t^{2k}}{(2k)!}(A^{2k})_{l+j,l}+\sum_{k\ge (|j|+1)/2}\f{t^{2k-1}}{(2k-1)!}(A^{2k-1})_{l+j,l},
\end{equation}
which, together with the second part of \eqref{ineq211} and \eqref{ineq213} as well as \eqref{ineqnew1}, implies that

\begin{eqnarray}\label{ineq214}
(e^{tA})_{i+l,i} &&\le C_1(\a\b)^{-\f{1}{2}} \sum_{k\ge \f{|l|}{2}} \f{t^{2k}}{k!}(\a\b)^k + C_1 t \sum_{k\ge \f{|l|+1}{2}} \f{t^{2k-1}}{(k-1)!}(\a\b)^{k-1} \nonumber\\
&&= C_2  \sum_{k\ge \f{|l|}{2}} \f{t^{2k}}{k!}(\a\b)^k+ C_1 t^2 \sum_{k\ge \f{|l|-1}{2}} \f{t^{2k}}{k!}(\a\b)^{k} .
\end{eqnarray}
The proof is complete. 
\end{proof}

\begin{lemma}\label{otoo}
$e^{tA}:\mc{O}\to \mc{O}$.
\end{lemma}
\begin{proof}
It suffices to show that $e^{tA}[\phi]\in \mc{O}$ for any $\phi\in \mc{O}$, that is, $\sum_{j\in \Z} (e^{tA})_{i,j}\phi(j)<+\infty$ uniformly in $i\in\Z$. Since $\phi\in \mc{O}$ is uniformly bounded, it remains to prove that there exists $C_3=C_3(t)>0$ (independent of $i\in\Z$) such that $\sum_{j\in \Z} (e^{tA})_{i,j}<C_3$. Indeed, by the upper bound obtained in Lemma \ref{uppbd} we can infer that
\begin{eqnarray}
\sum_{j\in\Z}(e^{tA})_{i,j}&& =\sum_{l\in\Z} (e^{tA})_{i,i+l}=  \sum_{l\in\Z} (e^{tA})_{i+l,i}\nonumber\\
&& \le 1+C_1t e^{\alpha\beta t^2}+C_2(e^{\alpha\beta t^2}-1)+ (C_2+C_1 t^2) \sum_{l\in\Z} \sum_{k\ge \f{|l|-1}{2}} \f{t^{2k}}{k!}(\a\b)^{k}\nonumber\\
&&= 1+C_1t e^{\alpha\beta t^2}+C_2(e^{\alpha\beta t^2}-1)+ (C_2+C_1 t^2) \sum_{k\ge 0} (4k+2) \f{t^{2k}}{k!}(\a\b)^{k}\nonumber\\
&& = 1+C_1t e^{\alpha\beta t^2}+C_2(e^{\alpha\beta t^2}-1)+ (C_2+C_1 t^2)(4\alpha\beta t^2+2)e^{\alpha\beta t^2}\nonumber\\
&&:=C_3(t).
\end{eqnarray}
Clearly, $C_3(t)$ is independent of $i\in\Z$. 
\end{proof}

\begin{lemma}\label{continuity}
There exists $C_4(t)$ such that 
\begin{equation}\label{c4}
\text{$C_4(t)$ is continuous and increasing in $t> 0$ with $\lim_{t\to 0} C_4(t)=1$}
\end{equation}
and 
\begin{equation}
\|e^{tA}\phi\|_\mc{O}\le C_4(t)\|\phi\|_{\mc{O}},\quad \forall \phi\in \mc{O}.
\end{equation}
\end{lemma}
\begin{proof}
Clearly, $e^{tA}$ is linear. It then suffices to prove that the norm of $e^{tA}$ is bounded by some $C_4(t)>0$ with  $\lim_{t\downarrow 0} C_4(t)=1$. Indeed, by the definition of $\|\cdot \|_{\mc{O}}$ and Fatou's Lemma, we have
\begin{eqnarray}
\|e^{tA}[\phi]\|_{\mc{O}}&&=\sum_{k\ge 1}2^{-k}\max_{|i|\le k}|\sum_{j\in\Z} (e^{tA})_{i,i+j}\phi(i+j)|\nonumber\\
&&\le \sum_{k\ge 1}2^{-k}\sum_{j\in\Z} \max_{|i|\le k}\left\{(e^{tA})_{i,i+j}|\phi(i+j)|\right\}\nonumber\\
&&: =\sum_{k\ge 1}2^{-k}\sum_{j\in\Z} I(j).
\end{eqnarray}
To estimate the term $I(j)$, we introduce variable change $l=i+j$. Consequently, \begin{equation}
I(j)=\max_{|i|\le k}\left\{(e^{tA})_{i,i+j}|\phi(i+j)|\right\}=\max_{|l-j|\le k}\left\{(e^{tA})_{l-j,l}|\phi(l)|\right\}\le \max_{|l|\le k+|j|}\left\{(e^{tA})_{l-j,l}|\phi(l)|\right\},
\end{equation}
which, combining with Lemma \ref{uppbd}, implies that
\begin{equation}
I(0)\le \left[1+C_1t e^{\alpha\beta t^2}+C_2(e^{\alpha\beta t^2}-1)\right]\max_{|l|\le k} |\phi(l)|
\end{equation}
and 
\begin{equation}
I(j)\le \left[C_1 t^2 \sum_{m\ge \f{|j|-1}{2}} \f{t^{2m}}{m!}(\a\b)^{m}+C_2 \sum_{m\ge  \f{|j|}{2}} \f{t^{2m}}{m!}(\a\b)^{m} \right]  \max_{|l|\le k+|j|} |\phi(l)|,  \quad j\neq 0.
\end{equation}
Hence, 
\begin{eqnarray}
&&\sum_{k\ge 1}2^{-k}\sum_{j\in\Z\setminus\{0\}} I(j)\nonumber\\
 && \le 2C_1 t^2 \sum_{k\ge 1}2^{-k}\sum_{j\ge 1} \sum_{m\ge \f{j-1}{2}} \f{t^{2m}}{m!}(\a\b)^{m}\max_{|l|\le k+j} |\phi(l)|+ 2C_2\sum_{k\ge 1}2^{-k}\sum_{j\ge 1} \sum_{m\ge \f{j}{2}} \f{t^{2m}}{m!}(\a\b)^{m}\max_{|l|\le k+j} |\phi(l)|\nonumber\\
&&:=2C_1 t^2 S_1+2C_2 S_2.
\end{eqnarray}
We first derive an upper bound for $S_1$. It is then similar for $S_2$.  Introducing the variable change $\tilde{j}=k+j$ and dropping the tilde, we obtain
\begin{equation}
S_1= \sum_{k\ge 1}2^{-k}\sum_{j\ge k+1} \sum_{m\ge \f{j-k-1}{2}} \f{t^{2m}}{m!}(\a\b)^{m}\max_{|l|\le j} |\phi(l)|,
\end{equation}
for which we use Fubini's theorem to the order of sum, yielding
\begin{equation}
S_1=\sum_{j\ge 2}\sum_{k=1}^{j-1}2^{-k}\sum_{m\ge \f{j-k-1}{2}} \f{(\a\b t^2)^{m} }{m!} \max_{|l|\le j}|\phi(l)|.
\end{equation}
Introducing the variable change $\tilde{k}=j-k-1$ and dropping the tilde, we obtain
\begin{eqnarray}
S_1= \sum_{j\ge 2}\sum_{k=0}^{j-2} 2^{-(j-k-1)}\sum_{m\ge k/2} \f{(\a\b t^2)^{m} }{m!} \max_{|l|\le j}|\phi(l)| \le 2 \left(\sum_{k\ge 0} 2^k\sum_{m\ge k/2} \f{(\a\b t^2)^m }{m!}\right)     \|\phi\|_{\mc{O}}.
\end{eqnarray}
To estimate the term in the bracket, we exchange the sum order to obtain
\begin{equation}
\sum_{k\ge 0} 2^k\sum_{m\ge k/2} \f{(\a\b t^2)^{m} }{m!}= \sum_{m\ge 0}\sum_{k=0}^{2m}  2^{k}\f{(\a\b t^2)^{m} }{m!}\le 2 e^{4\a\b t^2},
\end{equation}
where we have used the inequality $\sum_{k=0}^{2m}  2^{k}\le 2^{2m+1}$. Therefore, 
\begin{equation}
S_1 \le  4 e^{4\a\b t^2}  \|\phi\|_{\mc{O}}.
\end{equation}
Similarly, $S_2\le 2 \left(\sum_{k\ge 1} 2^k\sum_{m\ge k/2} \f{(\a\b t^2)^m }{m!}\right)     \|\phi\|_{\mc{O}}$ and hence, 
\begin{equation}
S_2\le  4 \left(e^{4\a\b t^2} -1\right) \|\phi\|_{\mc{O}}.
\end{equation}
To conclude, 
\begin{eqnarray}
&& \|e^{tA}[\phi]\|_{\mc{O}}\nonumber\\ && =\sum_{k\ge 1}2^{-k}\sum_{j\in\Z} I(j)= \sum_{k\ge 1}2^{-k}\left[ I(0)+\sum_{j\in\Z\setminus\{0\}} I(j)\right]
\nonumber\\
&& \le   \left[1+C_1t e^{\alpha\beta t^2}+C_2(e^{\alpha\beta t^2}-1)\right] \|\phi\|_{\mc{O}}+ 2C_1 t^2 S_1+2C_2S_2\nonumber\\
&&\le   \left[1+(C_1t+8C_1 t^2) e^{\alpha\beta t^2}+9C_2(e^{\alpha\beta t^2}-1) \right] \|\phi\|_{\mc{O}}  \nonumber\\
&&:=C_4(t) \|\phi\|_{\mc{O}}, 
\end{eqnarray}
where 
\begin{equation}
C_4(t):=1+(C_1t+8C_1 t^2) e^{\alpha\beta t^2}+9C_2(e^{\alpha\beta t^2}-1).
\end{equation} 
Clearly,  $C_4(t)$ satisfies \eqref{c4}.
\end{proof}

{\it Proof of Theorem \ref{fundamental}.} From Lemmas \ref{otoo} and \ref{continuity} we see that $e^{tA}$ is a well-defined linear operator from $\mc{O}$ to $\mc{O}$ and $e^{tA}[\phi]$ is jointly continuous in $(t,\phi)$. It then remains to verify that $\f{d}{dt} \left(e^{tA}[\phi]\right)=A[e^{tA}\phi]$ for any $\phi\in\mc{O}$. Indeed, for any $T>0$, 
\begin{equation}\label{ineq0915}
\sum_{j\in\Z}\sum_{n\ge 0} \f{d}{dt}\left\{\f{t^n(A^n)_{i,j}}{n!}\phi(j)\right\}=\sum_{j\in\Z}\sum_{n\ge 0} \f{t^n(A^{n+1})_{i,j}}{n!}\phi(j)=\sum_{j\in\Z}(Ae^{tA})_{i,j}\phi(j)=Ae^{tA}[\phi](i)
\end{equation}
uniformly in $t\in [0,T]$ and $i\in\Z$. Then by Fubini's theorem we may exchange the order of derivative and limits to obtain
\begin{equation}
 \f{d}{dt} \left\{\sum_{j\in\Z}\sum_{n\ge 0}\f{t^n(A^n)_{i,j}}{n!}\phi(j)\right\}=\sum_{j\in\Z}\sum_{n\ge 0} \f{d}{dt}\left\{\f{t^n(A^n)_{i,j}}{n!}\phi(j)\right\},
\end{equation}
which, combining with \eqref{ineq0915} implies that
\begin{equation}
\f{d}{dt} \{e^{tA}[\phi](i)\}=Ae^{tA}[\phi](i)=A[e^{tA}[\phi]](i), \quad \forall i\in\Z, t>0.
\end{equation} 
Therefore, $e^{tA}$, as defined in \eqref{def-etA},  is the fundamental solution matrix of $U'=AU$.

\setcounter{equation}{0}
\section{Proof of Theorem \ref{sstw}}

We apply the dynamical system theory developed in \cite{LZ-JFA10} to prove Theorem \ref{sstw}. For this purpose, we first show that \eqref{abstract-eq} generates a solution semiflow fitting the framework in \cite[Theorems 5.2 and 5.3]{LZ-JFA10} and \cite[Theorem 3.10]{LZ-CPAM07}.

By Theorem \ref{fundamental}, we can rewrite \eqref{abstract-eq} as the following integral equation. 
\begin{equation}\label{integral-eq}
\left\{
\begin{array}{ll}
U(t)= e^{(A+B)t}U(0)+\int_{0}^{t}e^{(A+B)(t-s)}F(U(s-\tau))ds, & t>0,\\
U(t)=\phi(t), &t \in [-\tau, 0],
\end{array}
\right.
\end{equation}
where $\phi\in \mc{C}$ is the given initial value\footnote{Note that $\mc{C}=C(\Z,X)$ and $X=C([-\tau,0],\R)$ as defined in section 3. Here by $\phi(t)$ we mean that $\phi(\cdot)(t)$.}. In \eqref{integral-eq} one may directly solve $U(t)$ for $t\in(0,\tau]$, and inductively for $t\in (n\tau,(n+1)\tau], n\ge 0$. For $t\ge 0$, define $Q_t: \mc{C}\to\mc{C}$ by 
\begin{equation}\label{def-flow}
Q_t[\phi](\theta):=U(t+\theta).
\end{equation}
For integers $a, b$ with $a<b$, we use the interval $[a,b]_{\Z}$ to denote the set of all integers between $a$ and $b$ (including $a$ and $b$). For $\phi\in\mc{C}$, we define $\phi_{[a,b]_\Z} \in C([a,b]_Z, X)$ by $\phi_{[a,b]_\Z}(i)=\phi(i), i\in[a,b]_\Z$. For any bounded set $\mc{U}\subset \mc{C}$, we define $\mc{U}_{[a,b]_\Z}:=\{\phi_{[a,b]_\Z}: \phi\in \mc{U}\}$.  Let $\kappa(\mc{U}_{[a,b]_\Z})$ be the Kuratowski mesasure of non-compactness for $\mc{U}_{[a,b]_\Z}$.
\begin{lemma}\label{semiflow}
$\{Q_t\}_{t\ge 0}$ is a monotone semiflow on $\mc{C}_{U^*}$ with the following properties: 
\begin{enumerate}
\item[(i)] $0$ and $U^*$ are the only $2$-periodic equilibria in $\mc{C}_{U^*}$. 
\item[(ii)] $Q_t[\phi](i+2)=Q_t[\phi(\cdot+2)](i)$ for  any $t\ge 0, i\in\Z, \phi\in \mc{C}_{U^*}$.
\item[(iii)] For any $I:=[0,a]_\Z$ with $a\in\Z_+$, there exists $\gamma(t)\in(0,1)$ such that 
\begin{equation}
\kappa (Q_t[\mc{U}])_I\le \gamma  \kappa (\mc{U}_I),\quad \forall   \mc{U}\subset \mc{C}_{U^*}.
\end{equation}
\end{enumerate}
\end{lemma}
\begin{proof}
We first prove that $Q_t[\phi]$ is continuous in $(t,\phi)$. By \eqref{integral-eq} and \eqref{def-flow}, we have
\begin{equation}\label{integral-abeq}
Q_t[\phi](\theta)=
\begin{cases}
e^{(A+B)(t+\theta)}\phi(0)+\int_{0}^{t+\theta}e^{(A+B)(t+\theta-s)}F(Q_s[\phi](-\tau))ds, & t+\theta>0,\\
\phi(t+\theta),&t+\theta\le 0.
\end{cases}
\end{equation}
For any $T>0$, by Lemma \ref{continuity} we can take the norm in both sides of \eqref{integral-abeq} to obtain
\begin{equation}
e^{\delta t}\|Q_t[\phi]\|_\mc{C}\le \max\{C_4(T),e^{-\delta T}\}\|\phi\|_{\mc{C}}+p\int_0^t  C_4(t-s) e^{\delta s} \|Q_s[\phi]\|_{\mc{C}}ds, \quad \forall t\in(0,T],
\end{equation}
where $\delta:=\max\{2\alpha+\eta, 2\beta+\gamma\}$ and $C_4(t)$ is as in \eqref{c4}. By using the Gronwall inequality we obtain
\begin{equation}
\|Q_t[\phi]\|_\mc{C}\le \max\{C_4(T),e^{-\delta T}\} e^{-\delta t+\int_0^t pC_4(s)ds}\|\phi\|_{\mc{C}},\quad t\in(0,T].
\end{equation}
Further, by the triangle inequality we can obtain the continuity of $Q_t[\phi]$ in $(t,\phi)\in\R_+\times \mc{C}_{U^*}$. 

Next we prove (iii).  We employ the same idea as in \cite[section 4]{LZ-JFA10}. Define 
\begin{equation}
L(t)[\phi](\theta)=
\begin{cases}
\phi(t+\theta)-\phi(0), &t+\theta<0\\
0,&t+\theta\ge 0
\end{cases}
\quad 
S(t)[\phi](\theta)=
\begin{cases}
\phi(0), &t+\theta<0\\
U(t+\theta),&t+\theta\ge 0.
\end{cases}
\end{equation}
Then $Q_t[\phi]=L(t)[\phi]+S(t)[\phi]$. By the same argument as in  \cite[section 4]{LZ-JFA10} we see that for any $\mc{U}\in \mc{C}_{U^*}$, the set $(S(t)[\mc{U}])_I$ is compact in $C(I, X)$ and $\kappa(L(t)[\mc{U}])_I\le e^{-\delta t} \kappa (\mc{U}_I)$ for some $\delta>0$. Thus, 
\begin{equation}
\kappa((Q_t[\mc{U}])_{I})\le \kappa((L(t)[\mc{U}])_{I})+\kappa((S(t)[\mc{U}])_{I})= \kappa((L(t)[\mc{U}])_{I})\le e^{-\delta t} \mc{U}_I. 
\end{equation}
Thus, statement (iii) is proved. 

Other statements are obvious and we omit the details. 
\end{proof}

{\bf Proof of Theorem \ref{sstw}:} In \cite[Theorems 5.2 and 5.3]{LZ-JFA10} we choose $\mc{H}=\Z, \tilde{H}=2\Z, \beta=U^*, \mc{M}=Y=\mc{C}_{U^*}$.
Then by Lemma \ref{semiflow} we see that the solution semiflow $\{Q_t\}_{t\ge 0}$ satisfies all conditions there. Further, by the symmetry and the sublinearity of the semiflow we obtain the existence of the leftward and rightward spreading speed $c^*\ge 0$ that coincides with the minimal speed of pulsating waves,  in the sense of  \cite[Theorems 5.2 and 5.3]{LZ-JFA10}. 

Next we give the variational characterization of $c^*$. Indeed, by using \cite[Theorem 3.10]{LZ-CPAM07}, we can infer that $c^*=\inf_{\mu>0}\f{\lambda(\mu)}{\mu}$, where $\lambda(\mu)$ is the principle eigenvalue of the following problem
\begin{equation}\label{eigen1}
\begin{cases}
\lambda e^{\mu \cdot}\phi= (A+B+e^{-\lambda\tau}DF(0))[e^{\mu\cdot}\phi]\\
\phi(i)=\phi(i+2)>0, i\in\Z.
\end{cases}
\end{equation}
In virtue of the properties of operators $A, B$ and $DF(0)$, we compute to have an equivalent eigenvalue problem
\begin{equation}\label{eigen2}
\begin{cases}
\lambda \phi(0)=\alpha (e^\mu+e^{-\mu})\phi(1)+(-2\beta-\gamma+f'(0)e^{-\lambda\tau})\phi(0)\\
\lambda \phi(1)=\beta (e^\mu+e^{-\mu})\phi(0)+(-2\alpha+\eta)\phi(1).
\end{cases}
\end{equation}
Since $\phi(i)>0, i\in\Z$ as assumed, we further simplified \eqref{eigen2} by solving its second equation, yielding that $\lambda=\lambda(\mu)$ is the unique positive solution of
\begin{equation}\label{eigen3}
F(\lambda,\mu,\beta)=0,
\end{equation}
where 
\begin{equation}
F(\lambda,\mu,\beta):=-(\lambda+2\beta+\gamma)+\f{\alpha\beta(e^{\mu}+e^{-\mu})^2}{\lambda+2\alpha+\eta}+f'(0)e^{-\lambda\tau}.
\end{equation}

Finally we show that $c^*>0$. Indeed, since $\lim_{\mu\downarrow 0}\lambda(\mu)=\lambda(0)>0$ is the unique positive solution of $F(\lambda, 0,\beta)$ and $\lambda(\mu,\beta)\sim \alpha\beta e^\mu$ as $\mu\to\infty$, we can infer that $c^*=\inf_{\mu>0}\f{\lambda(\mu,\beta)}{\mu}$ is attained at some $\mu^*=\mu^*>0$. Therefore, $c^*>0$. The proof is complete.

\setcounter{equation}{0}
\section{Proof of Theorems \ref{ss1} and \ref{ss2}}
In the previous section, we have established $c^*=\inf_{\mu>0}\f{\lambda(\mu)}{\mu}$, where $\lambda(\mu)$ is the unique positive zero of $F(\lambda,\mu,\beta)$. In this section, to investigate the influence of $\beta$ we write $c^*(\beta)$ and  $\lambda(\mu,\beta)$ instead of $c^*$ and $\lambda(\mu)$, respectively.

Before proving Theorem \ref{ss1}, we first ensure that the maximum of $c^*=c^*(\beta)$ exists when $\beta$ varies in $(0,\beta_0)$.

\begin{lemma}\label{limit1}
$\lim_{\beta\downarrow 0}c^*(\beta)=0$.
\end{lemma}
\begin{proof}
By \eqref{eigen3} we infer that
\begin{equation}\label{ineq301}
(\lambda+2\alpha+\eta)(\lambda+2\beta+\gamma-f'(0))<\alpha\beta(e^{\mu}+e^{-\mu})^2.
\end{equation}
Note that $\beta\in(0,\beta_0)$. It then follows that there exists $\mu_0>0$ (independent of $\beta$) such that 
\begin{equation}
\lambda=\lambda(\mu,\beta)<2\sqrt{\alpha\beta}(e^{\mu}+e^{-\mu}),\quad \forall \mu\ge \mu_0, \beta\in(0,\beta_0).
\end{equation} 
Hence, 
\begin{equation}
c^*(\beta)\le 2\sqrt{\alpha\beta}\f{e^{\mu_0}+e^{-\mu_0}}{\mu_0},
\end{equation}
which implies that $\lim_{\beta\downarrow 0}  c^*(\beta)=0$.
\end{proof}

\begin{lemma}\label{limit2}
$\lim_{\beta\uparrow \beta_0}c^*(\beta)=0$.
\end{lemma}
\begin{proof}
For small $\mu$, $(e^\mu+e^{-\mu})^2=4+4\mu^2+o(\mu)$. Then there exists $\mu_0>0$ such that 
\begin{equation}
(e^\mu+e^{-\mu})^2<4+5\mu^2,\quad  \forall \mu\in(0,\mu_0),
\end{equation}
which, together with \eqref{ineq301}, implies that
\begin{equation}
(\lambda+2\alpha+\eta)(\lambda+2\beta+\gamma-f'(0))<\alpha\beta(4+5\mu^2),\quad \mu\in(0,\mu_0).
\end{equation}
Solving this inequality yields
\begin{equation}\label{ineq302}
\lambda\le \f{1}{2}(-b+\sqrt{b^2- 4c})=\f{-2c}{b+\sqrt{b^2-4c}}
\end{equation}
with
\[
b=2\alpha+\eta+2\beta+\gamma-f'(0),\quad c=(2\alpha+\eta)(2\beta+\gamma-f'(0))-\alpha\beta(4+5\mu^2).
\]
By direct calculations and the definition of $\beta_0$ in \eqref{def-beta0}, we obtain
\begin{equation}
c=[\gamma-f'(0)](2\alpha+\eta)+2\eta\beta-5\alpha\beta\mu^2=-2\eta(\beta_0-\beta)-5\alpha\beta\mu^2<0,\quad \beta\in(0,\beta_0),
\end{equation}
from which we immediately see that 
\begin{equation}\label{ineq303}
\lambda\le \f{-c}{b}=\f{1}{b}\left( 2\eta(\beta_0-\beta)+5\alpha\beta\mu^2  \right),\quad \forall \mu\in(0,\mu_0),\beta\in (0,\beta_0).
\end{equation}
For aforementioned $\mu_0$, there exists $\beta_1\in(0,\beta_0)$ such that  
\begin{equation}
\sqrt{\f{2\eta(\beta_0-\beta)}{5\alpha\beta}}\in (0,\mu_0), \quad \beta\in (\beta_1,\beta_0).
\end{equation} 
As such, when $\beta\in(\beta_1,\beta_0)$ we choose in particular $\mu=\sqrt{\f{2\eta(\beta_0-\beta)}{5\alpha\beta}}$ in \eqref{ineq303}. Then we obtain
\begin{eqnarray}
c^*(\beta)=\inf_{\mu>0}\f{\lambda(\mu,\beta)}{\mu}\le \left. \f{\lambda(\mu,\beta)}{\mu}\right|_{\mu=\sqrt{\f{2\eta(\beta_0-\beta)}{5\alpha\beta}}}\le  \f{2}{b}\sqrt{10\alpha\beta\eta(\beta_0-\beta)},\quad \beta\in(\beta_1,\beta_0),
\end{eqnarray}
which implies that $\lim_{\beta\uparrow\beta_0}c^*(\beta)=0$.
\end{proof}

{\bf Proof of Theorem \ref{ss1}}:  Introducing the variable change $c=c(\mu,\beta)=\f{\lambda(\mu,\beta)}{\mu}$, we see that \eqref{eigen3} becomes
\begin{equation}\label{eigen4}
F(c\mu,\mu,\beta)=0.
\end{equation}
Since $\inf_{\mu>0}\f{\lambda(\mu,\beta)}{\mu}$ is attained at some $\mu^*>0$, we see that $(c^*,\mu^*)$,  depending on $\beta\in(0,\beta_0)$,  is determined by the following system of transcendental equations
\begin{equation}\label{eq401}
F(c\mu,\mu,\beta)=0, \quad \f{d}{d\mu} F(c\mu,\mu,\beta)=0.
\end{equation}
Next we employ the implicit function theorem to calculate $\f{d}{d\beta} c^*(\beta)$.  Indeed, if the matrix
\begin{equation}
J:=\left(
\begin{array}{cc}
\f{d}{dc} F(c\mu,\mu,\beta) & \f{d}{d\mu}F(c\mu,\mu,\beta)\\
\f{d^2}{dcd\mu} F(c\mu,\mu,\beta)  &  \f{d^2}{d\mu^2}F(c\mu,\mu,\beta) 
\end{array}
\right)|_{(c^*,\mu^*)}
\end{equation}
is invertible, then 
\begin{equation}\label{eq305}
\f{d}{d\beta} c^*(\beta)= -\f{\f{d}{d\beta} F(c\mu,\mu,\beta)}{\f{d}{dc}F(c\mu,\mu,\beta)}|_{(c^*,\mu^*)}
\end{equation}
thanks to $\f{d}{d\mu} F(c\mu,\mu,\beta)|_{(c^*,\mu^*)}=0$. In the following we check that $J$ is invertible. It suffices to verify that none of the diagonal entries  of $J$ is zero, thanks again to $\f{d}{d\mu} F(c\mu,\mu,\beta)|_{(c^*,\mu^*)}=0$. By direct computations, we have
\begin{equation}\label{eq402}
\f{d}{dc} F(c\mu,\mu,\beta)=\mu \partial_1 F |_{(c\mu,\mu,\beta)}, \quad   \f{d}{d\mu}F(c\mu,\mu,\beta)=(c\partial_1 F+\partial_2 F)|_{(c\mu,\mu,\beta)}.
\end{equation}
Since at $(c^*\mu^*, \mu^*,\beta)$ we have $ \f{d}{d\mu}F(c\mu,\mu,\beta)=0$, it then follows that
\begin{equation}\label{eq306}
\f{d}{dc} F(c\mu,\mu,\beta)|_{(c^*,\mu^*)}=-\f{\mu^*}{c^*} \partial_2 F |_{(c^*\mu^*,\mu^*,\beta)}=-\f{2\alpha\beta \mu^*}{c^*(c^*\mu^*+2\alpha+\eta)}(e^{2\mu^*}-e^{-2\mu^*})<0.
\end{equation}
Combining \eqref{eq401} and \eqref{eq402} we derive $c^*=-\f{\partial_2 F}{\partial_1 F}|_{(c^*\mu^*,\mu^*,\beta)}$, by which we further compute to obtain
\begin{equation}
 \f{d^2}{d\mu^2}F|_{(c^*\mu^*,\mu^*,\beta*)}= \f{1}{(\partial_1 F)^2}[(\partial_2 F)^2\partial_1^2 F-2\partial_2 F \partial_{12}^2 F\partial_1 F+\partial_{22}^2 F(\partial_1 F)^2]|_{(c^*\mu^*,\mu^*,\beta)}.
\end{equation}
Note that 
\begin{equation}
\partial_2 F=\f{2\alpha\beta(e^{2\mu}-e^{-2\mu})}{\lambda+2\alpha+\eta}, \quad \partial_{12}^2 F=-\f{2\alpha\beta(e^{2\mu}-e^{-2\mu})}{(\lambda+2\alpha+\eta)^2}, \quad \partial_{22}^2 F=\f{4\alpha\beta(e^{2\mu}+e^{-2\mu})}{\lambda+2\alpha+\eta}
\end{equation}
and
\begin{equation}
\partial_1^2 F=\tau^2f'(0)e^{-\lambda\tau}+\f{2\alpha\beta(e^{\mu}+e^{-\mu})^2}{(\lambda+2\alpha+\eta)^3}\ge \f{2\alpha\beta(e^{\mu}+e^{-\mu})^2}{(\lambda+2\alpha+\eta)^3}.
\end{equation}
It then follows that
\begin{eqnarray}
&&(2\partial_2 F \partial_{12}^2 F)^2-4\partial_{22}^2 F(\partial_2 F)^2\partial_1^2 F\nonumber\\
\le &&(2\partial_2 F \partial_{12}^2 F)^2-4\partial_{22}^2 F(\partial_2 F)^2 \f{2\alpha\beta(e^{\mu}+e^{-\mu})^2}{(\lambda+2\alpha+\eta)^3}\nonumber\\
=&&-\f{64\alpha^2\beta^2(e^{2\mu}-e^{-2\mu})^2}{(\lambda+2\alpha+\eta)^6}(e^{\mu}+e^{-\mu})^4\nonumber\\
<&&0,\quad \text{at $(c^*\mu^*,\mu^*,\beta)$},
\end{eqnarray}
from which we infer that the polynomial $\partial_{22}^2 F x^2-2\partial_2 F \partial_{12}^2 Fx+(\partial_2 F)^2\partial_1^2 F$ is positive for all $x\in\R$. Consequently, $\f{d^2}{d\mu^2}F|_{(c^*\mu^*,\mu^*,\beta*)}>0$. Thus, $J$ is invertible,  and hence, \eqref{eq305} holds. Combining \eqref{eq305} and \eqref{eq306} we see that the sign of $\f{d}{d\beta} c^*(\beta)$ is the same as that of $\f{d}{d\beta} F(c\mu,\mu,\beta)|_{(c^*,\mu^*)}$. Moreover, $\mu=\mu^*(\beta)$ is continuous in $\beta$, due to the implicit function theorem.

Define
\begin{equation}
\mc{B}:=\left\{\beta\in(0,\beta_0): \f{d}{d\beta} F(c\mu,\mu,\beta)|_{(c^*,\mu^*)}=0\right\}.
\end{equation}
By Lemmas \ref{limit1} and \ref{limit2} we see that $\mc{B}\neq\emptyset$. In the following, we show that $\mc{B}$ is a singleton. Indeed,
note that
\begin{equation}\label{eq309}
0=F(c^*\mu^*,\mu^*,\beta)=-(c^*\mu^*+2\beta+\gamma)+\f{\alpha\beta(e^{\mu^*}+e^{-\mu^*})^2}{c^*\mu^*+2\alpha+\eta}+f'(0)e^{-c^*\mu^*\tau},\quad \beta\in(0,\beta_0),
\end{equation}
which, combining with the equality
\begin{equation}\label{eq307}
0=\f{d}{d\beta} F(c\mu,\mu,\beta)|_{(c^*,\mu^*)}=-2+\f{\alpha(e^{\mu^*}+e^{-\mu^*})^2}{c^*\mu^*+2\alpha+\eta},\quad \beta\in \mc{B},
\end{equation}
implies that for any $\beta\in \mc{B}$ the number $\lambda^*(\beta):=c^*(\beta)\mu^*(\beta)$  is the unique solution of 
\begin{equation}\label{eq308}
-(\lambda+\gamma)+f'(0)e^{-\lambda\tau}=0.
\end{equation}
Thus, $\lambda^*=\lambda^*(\beta)$ is independent of $\beta\in\mc{B}$. To show that $\mc{B}$ is a singleton, we first observe from \eqref{eq402} that 
\begin{eqnarray*}
0&&=\f{d}{d\mu} F(c\mu,\mu,\beta)|_{(c^*,\mu^*)}=(c\partial_1 F+\partial_2 F)|_{(c^*\mu^*,\mu^*)}\nonumber\\
&&=\left\{-c^*\left(1+\tau f'(0)e^{-c^*\mu^*\tau}+\f{\alpha\beta(e^{\mu^*}+e^{-\mu^*})^2}{(c^*\mu^*+2\alpha+\eta)^2}\right)+\f{2\alpha\beta}{c^*\mu^*+2\alpha+\eta}(e^{2\mu^*}-e^{-2\mu^*})\right\},
\end{eqnarray*}
which, for $\beta\in\mc{B}$, can be simplified into the following equality thanks to the relations in \eqref{eq309} and \eqref{eq307}:
\begin{equation}
h(\mu^*(\beta))=C(\lambda^*, \beta),\quad \beta\in\mc{B},
\end{equation}
where
\begin{equation}
h(\mu)=\mu(e^{\mu}-e^{-\mu})
\end{equation}
and
\begin{equation}
C(\lambda, \beta)=\f{\lambda}{\sqrt{2\alpha(\lambda+2\alpha+\eta)}}+\f{1}{\beta\sqrt{8\alpha}}\lambda(1+\tau f'(0)e^{-\lambda\tau})\sqrt{\lambda+2\alpha+\eta}.
\end{equation}
Note that $h$ is strictly increasing in $\mu\ge 0$, so is its inverse $h^{-1}$. It then follows that
\begin{equation}\label{eq310}
\mu^*(\beta)=h^{-1}(C(\lambda^*,\beta)), \quad \beta\in\mc{B}.
\end{equation}
On the other hand, by \eqref{eq307} we have 
\begin{equation}\label{eq311}
\mu^*(\beta)=\cosh^{-1}\left(\sqrt{\f{\lambda^*+2\alpha+\eta}{2\alpha}}\right), \quad \beta\in \mc{B}.
\end{equation}
Therefore, combining \eqref{eq310} and \eqref{eq311} yields that
\begin{equation}\label{312}
h^{-1}(C(\lambda^*,\beta))=\cosh^{-1}\left(\sqrt{\f{\lambda^*+2\alpha+\eta}{2\alpha}}\right), \quad \beta\in \mc{B}.
\end{equation}
As such, any $\beta\in \mc{B}$ is the zero of the strictly decreasing function 
\begin{equation}\label{beta-eq}
h^{-1}(C(\lambda,\beta))-\cosh^{-1}\left(\sqrt{\f{\lambda^*+2\alpha+\eta}{2\alpha}}\right),\quad \beta\in(0,\beta_0).
\end{equation}
Hence, $\mc{B}=\{\beta_1\}$ for some $\beta_1\in(0,\beta_0)$, and 
\begin{equation}
c^*(\beta_1)=\f{\lambda^*(\beta_1)}{\mu^*(\beta_1)}=\f{\lambda^*}{\cosh^{-1}\left(\sqrt{\f{\lambda^*+2\alpha+\eta}{2\alpha}}\right)}.
\end{equation}
Moreover, since $c^*(\beta)>0$ for $\beta\in(0,\beta_0)$, by Lemmas \ref{limit1} and \ref{limit2} we obtain that $c^*(\beta_1)=\max_{\beta\in(0,\beta_0)}c^*(\beta)$. The proof is complete. 

{\bf Proof of Theorem \ref{ss2}}: Recall that $\lambda(\mu,\beta)$ is the unique positive zero of $F(\lambda, \mu,\beta)$, which is defined in \eqref{eigen3}. Since here the parameter $\eta$ is of interest, we write $\lambda(\mu,\eta)$ instead of $\lambda(\mu,\beta)$. Note that $F$ is decreasing in $\lambda$ and $\eta$. It follows that $\lambda(\mu,\eta)$ is decreasing in $\eta$, so is $c^*(\eta)$ due to $c^*(\eta)= \f{\lambda(\mu_1,\eta)}{\mu_1}$ for some $\mu_1=\mu_1(\eta)>0$. It then remains to show the limit is zero.

Next we proceed with two cases: (i) $f'(0)\in (\Gamma, 2\beta+\gamma)$; (ii) $f'(0)\ge 2\beta+\gamma$.

(i) Fix $\epsilon\in (0, \eta_0)$. For $\eta\in[\eta_0-\epsilon, \eta_0)$, by the monotonicity of $F$ in $\eta$ and $\lambda$, we have
\begin{equation}
\lambda(\mu,\eta)\le \lambda_0^\epsilon(\mu),
\end{equation}
where $\lambda_0^\epsilon(\mu)$ is the unique positive solution of 
\begin{equation}
-(\lambda+2\beta+\gamma)+\f{\alpha\beta(e^\mu+e^{-\mu})^2}{\lambda+2\alpha+\eta_0-\epsilon}+f'(0)=0.
\end{equation}
Since $f'(0)\in (\Gamma, 2\beta+\gamma)$, by the definition of $\eta_0$ we obtain the relation $f'(0)=\gamma+\f{2\beta\eta_0}{2\alpha+\eta_0}$, by which we further compute to obtain
\begin{equation}
\lambda_0^\epsilon(\mu)=\f{2a_2}{a_1+\sqrt{a_1^2+4a_2}}, 
\end{equation}
where $a_1=2\alpha+\eta_0-\epsilon+\f{4\alpha\beta}{2\alpha+\eta_0}$ and $a_2=\alpha\beta[(e^\mu-e^{-\mu})^2+\f{4\epsilon}{2\alpha+\eta_0}]$. Note that $e^{\mu-\e^{-\mu}}\le 4\mu$ for $\mu\in (0,1)$. It then follows that 
\begin{equation}
\lambda_0^\epsilon(\mu)\le \f{a_2}{a_1}\le \f{\alpha\beta}{a_1} \left(16\mu^2+\f{4\epsilon}{2\alpha+\eta_0}\right),\quad \forall \mu\in(0,1).
\end{equation}
As such, for any $\mu\in(0,1)$ and $\epsilon\in (0,\eta_0)$, we have
\begin{equation}
\lim_{\eta\uparrow \eta_0}c^*(\eta)\le c^*(\eta_0-\epsilon)=\inf_{\mu\in (0,1)}\f{\lambda(\mu, \eta_0-\epsilon)}{\mu} \le \f{\lambda_0^\epsilon(\mu)}
{\mu}\le    \f{\alpha\beta}{a_1} \left(16\mu+\f{4\epsilon}{\mu(2\alpha+\eta_0)}\right).
\end{equation}
Since $\mu$ and $\epsilon$ can be independently arbitrary small, it then follows that $\lim_{\eta\uparrow \eta_0}c^*(\eta)=0$.

(ii) By inequality \eqref{ineq301} we can infer that $\lambda(\mu, \eta)\le \lambda_1(\mu,\eta)$, where $\lambda_1(\mu,\eta)$ is the positive solution of $(\lambda+2\alpha+\eta)(\lambda+2\beta+\gamma-f'(0))=\alpha\beta(e^\mu+e^{-\mu})^2$. Note that $\lambda_1(\mu,\eta)$ decreases in $\eta$ to $f'(0)-2\beta+\gamma$ as $\eta\to\infty$ for any $\mu>0$. It then follows that
\begin{equation}
\lim_{\eta\to\infty}c^*(\eta)=\lim_{\eta\to\infty}\inf_{\mu>0}\f{\lambda(\mu,\eta)}{\mu}\le \lim_{\eta\to\infty}\f{\lambda(\mu,\eta)}{\mu}\le \f{f'(0)-2\beta+\gamma}{\mu},\quad \forall \mu>0,
\end{equation}
in which passing $\mu\to+\infty$ we obtain $\lim_{\eta\to\infty}c^*(\eta)=0$. The proof is complete.

\bibliographystyle{siam}

\end{document}